\documentclass[a4paper,12pt,reqno]{amsart}
\usepackage[a4paper,margin=1.9cm,top=2.5cm,bottom=2.5cm,centering,vcentering]{geometry}
\usepackage{hyperref,graphicx,xcolor}
\usepackage{amsmath}
\theoremstyle{plain}
\newtheorem{theorem}{Theorem}[section]

\newtheorem{corollary}[theorem]{Corollary}
\theoremstyle{definition}
\newtheorem{definition}{Definition}

\numberwithin{equation}{section}

\newcommand{\la}{\lambda}
\newcommand{\vin}{\vdash}
\newcommand{\qq}{(q;q)_\infty}
\newcommand{\gauss}{\genfrac{[}{]}{0pt}{}}

\numberwithin{equation}{section}

\begin{document}

\title{Refinement of some partition identities of Merca and Yee}
\author[P. J. Mahanta]{Pankaj Jyoti Mahanta}
\address{Gonit Sora, Dhalpur, Assam 784165, India}
\email{pankaj@gonitsora.com}
\author[M. P. Saikia]{Manjil P. Saikia}
\address{School of Mathematics, Cardiff University, Cardiff, CF24 4AG, UK}
\email{manjil@saikia.in}



\date{\today.}

\begin{abstract}
    Recently, Merca and Yee proved some partition identities involving two new partition statistics.  We refine these statistics and generalize the results of Merca and Yee. We also correct a small mistake in a result of Merca and Yee.
\end{abstract}

\maketitle

\section{Introduction}

A partition of an integer $n$, is a sequence of weakly decreasing positive integers such that they sum up to $n$. The terms of the sequence are called parts and a partition $\lambda$ of $n$ is denoted by $\lambda \vdash n$. We denote by $p(n)$, the number of partitions of $n$. For instance, $2+2+1$ is a partition of $5$ and $p(5)=7$. A masterful treatment of this topic is in the book by Andrews \cite{AndrewsBook}.

There is a rich history and literature on partitions with various statistics attached to them. Recently, Merca and Yee \cite{MercaYee} studied several such statistics and proved several interesting results (both analytially and combinatorially). The aim of this paper is to refine the results of Merca and Yee \cite{MercaYee} by putting in additional constraints on the partition statistics they studied.

The following functions are of interest in this paper.
\begin{definition}
For a positive integer $n$, we define
\begin{enumerate}
    \item $a_k(n)$ to be the sum of the parts which are divisible by $k$ counted without multiplicity in all the partitions of $n$,
    \item $a_{k,p}(n)$ to be the sum of the parts which are congruent to $p \pmod k$ counted without multiplicity in all the partitions of $n$, where $0\leq p\leq k-1$, and
    \item $b_k(n)$ to be the sum of the distinct parts of $n$ that appear at least $k$ times in all the partitions of $n$.
\end{enumerate}
\end{definition}
For example, $a_3(5)=6$, $a_{3,0}(5)=6$, $a_{3,1}(5)=9$, $a_{3,2}(5)=11$ and $b_3(5)=2$, which can be seen from the fact that the partitions of $5$ are \[5, 4+1, 3+2, 3+1+1, 2+2+1, 2+1+1+1, 1+1+1+1+1.\]

Merca and Yee \cite{MercaYee} studied related functions. In particular, they studied $a(n)$, the sum of parts counted without multiplicity in all the partitions of $n$ and $b(n)$, the sum of distinct parts that appear at least $2$ times in all the partitions of $n$. It is clear from the definition that \[
a(n)=\sum_{p=0}^{k-1}a_{k,p}(n),\] and $b_2(n)=b(n)$. So, $a_{k,p}(n)$ and $b_k(n)$ can be said to be refinements of $a(n)$ and $b(n)$. They also studied the function $a_{2,0}(n)$ and $a_{2,1}(n)$ which they denoted by $a_e(n)$ and $a_o(n)$ respectively. We will keep their notation for these special cases in the remainder of this paper.

Merca and Yee \cite{MercaYee} found the generating functions of $a(n)$, $a_e(n)$, $a_o(n)$ and $b(n)$, connected these functions in terms of very simple relations and then further connected the function $b(n)$ to two other partition functions $M_\ell (n)$ and $MP_\ell(n)$, which we will define in the next section. The aim of the present paper is to generalize all of these results for our refined functions $a_{k,p}(n)$ and $b_k(n)$. While doing this, we also correct a minor error in a result of Merca and Yee \cite{MercaYee}.

The rest of the paper is organized as follows: in Section \ref{sec:results} we state all of our results and show as corollaries all of the results of Merca and Yee \cite{MercaYee}, in Section \ref{sec:proof1} we prove our results using analytical techniques, in Section \ref{sec:comb} we prove all but one of our results using combinatorial techniques, and finally we end the paper with some remarks in Section \ref{sec:rem}. We closely follow the techniques used by Merca and Yee \cite{MercaYee} in our proofs.

\section{Results and Corollaries}\label{sec:results}

We need the notation for the $q$-Pochhammer symbol
\[
(a;q)_\infty=\prod_{n=0}^\infty (1-aq^n) \quad \text{for}~|q|<1.
\]The generating functions for $a_k(n)$, $a_{k,p}(n)$ and $b_k(n)$ are given in the following theorem.
\begin{theorem}\label{ThmGF}
We have
\begin{align*}
     \sum_{n=1}^\infty a_k(n)q^n &=~\frac{1}{(q;q)_\infty}\cdot \frac{kq^k}{(1-q^k)^2},\\
          \sum_{n=1}^\infty a_{k,p}(n)q^n &=~\frac{1}{(q;q)_\infty}\cdot \frac{(pq^{p-k}+(k-p)q^p)q^k}{(1-q^k)^2},\\
               \sum_{n=1}^\infty b_k(n)q^n &=~\frac{1}{(q;q)_\infty}\cdot \frac{q^k}{(1-q^k)^2}.
\end{align*}
\end{theorem}
\noindent From the above theorem (as well as combinatorially, which we will prove later) the following result follows.
\begin{theorem}\label{ThmComb}
For all $n\geq 1$, we have
\begin{enumerate}
    \item $a_k(n)=kb_k(n)$, and
    \item $a_{k,p}(n)=(k-p)b_k(n-p)+pb_k(n+k-p)$.
\end{enumerate}
\end{theorem}

As easy corollaries of the above results, two results of Merca and Yee \cite{MercaYee} follow.
\begin{corollary}[Theorem 1.2, \cite{MercaYee}]\label{CorM}
We have
\begin{align*}
         \sum_{n=1}^\infty a_e(n)q^n =~\sum_{n=1}^\infty a_{2,0}(n)q^n &=~\frac{1}{(q;q)_\infty}\cdot \frac{2q^2}{(1-q^2)^2},\\
              \sum_{n=1}^\infty a_o(n)q^n =~\sum_{n=1}^\infty a_{2,1}(n)q^n &=~\frac{1}{(q;q)_\infty}\cdot \frac{q(1+q^2)}{(1-q^2)^2},
              \end{align*}
              and
              \[
                        \sum_{n=1}^\infty a(n)q^n =~\frac{1}{(q;q)_\infty}\cdot \frac{q}{(1-q)^2}.
\]
\end{corollary}
\begin{corollary}[Theorem 1.3, \cite{MercaYee}]
For all $n\geq 1$, we have
\begin{enumerate}
    \item $a_e(n)=a_{2,0}(n)=2b(n)$,
    \item $a_o(n)=a_{2,1}(n)=b(n+1)+b(n-1)$, and
    \item $a(n)=a_{2,0}(n)+a_{2,1}(n)=b(n+1)+2b(n)+b(n-1)$.
\end{enumerate}
\end{corollary}

Andrews and Merca \cite{AndrewsMerca1} introduced a new partition function $M_\ell(n)$, which counts the number of partitions of $n$ where $\ell$ is the least positive integer that is not a part and there are more parts which are greater than $\ell$ than there are parts less than $\ell$. For instance $M_3(5)=0$. We can connect this function with $b_k(n)$ in the following way.
\begin{theorem}\label{Thm:trunc}
For any positive integer $k, \ell$ and $n$, we have
\begin{multline*}
    (-1)^{\ell-1}\left(\sum_{j=-(\ell-1)}^\ell (-1)^j b_k(n-j(3j-1)/2)-\frac{1+(-1)^{[n \equiv 0 \pmod k]+1}}{2}\cdot \frac{n}{k} \right)\\=\sum_{j=1}^{\lfloor n/k\rfloor}jM_\ell(n-kj),
\end{multline*}
where we have used the Iverson bracket, $[P]$ which returns the value $1$ if the logical proposition $P$ is satisfied, and returns $0$ otherwise.
\end{theorem}
The following are two easy corollaries of the above theorem.
\begin{corollary}\label{cor-m-t}
For any positive integer $k, \ell$ and $n$, we have
\[
    (-1)^{\ell-1}\left(\sum_{j=-(\ell-1)}^\ell (-1)^j b_k(n-j(3j-1)/2)-\frac{1+(-1)^{[n \equiv 0 \pmod k]+1}}{2}\cdot \frac{n}{k} \right)\geq 0.
\]
\end{corollary}
\begin{corollary}
For any positive integer $k$ and $n$, we have
\[
\sum_{j=-\infty}^\infty (-1)^j b_k(n-j(3j-1)/2)=\frac{1+(-1)^{[n \equiv 0 \pmod k]+1}}{2}\cdot \frac{n}{k}.
\]
\end{corollary}
From the above theorem and corollaries, the following results follow easily.
\begin{corollary}[Theorem 1.4, \cite{MercaYee}]
For any positive integers $\ell$ and $n$, we have
\[
    (-1)^{\ell-1}\left(\sum_{j=-(\ell-1)}^\ell (-1)^j b(n-j(3j-1)/2)-\frac{1+(-1)^{n}}{2}\cdot \frac{n}{2} \right)=\sum_{j=1}^{\lfloor n/2\rfloor}jM_\ell(n-2j).
\]
\end{corollary}
\begin{corollary}[Corollary 1.5, \cite{MercaYee}]
For any positive integers $\ell$ and $n$, we have
\[
(-1)^{\ell-1}\left(\sum_{j=-(\ell-1)}^\ell (-1)^j b(n-j(3j-1)/2)-\frac{1+(-1)^{n}}{2}\cdot \frac{n}{2} \right)\geq 0.
\]
\end{corollary}
\begin{corollary}[Corollary 1.6, \cite{MercaYee}]
For any positive integer $n$, we have
\[
\sum_{j=-\infty}^\infty (-1)^j b(n-j(3j-1)/2)=\frac{1+(-1)^{n}}{2}\cdot \frac{n}{2}.
\]
\end{corollary}

Andrews and Merca \cite{AndrewsMerca2} studied a new partition function $MP_\ell (n)$, which counts the number of partitions of $n$ in which the first part larger than $2k-1$ is odd and appears exactly $k$ times, and all other parts appear at most one time. For instance, $MP_3(5)=3$. We can connect the function $b_k(n)$ with $MP_\ell(n)$ using a new function $c_k(n)$ in the following way.

\begin{theorem}\label{Thm:gen1.7}
	For any positive integer $k, \ell$ and $n$, we have
	\begin{multline*}
	    	(-1)^{\ell-1}\left(\sum_{j=0}^{2\ell-1} (-1)^{\frac{j(j+1)}{2}} b_k(n-j(j+1)/2)-\frac{1+(-1)^{[n \equiv 0 \pmod k]+1}}{2}\cdot c_k(n) \right)\\=\sum_{j=0}^n c_k(j)MP_\ell(n-j),
	\end{multline*}
	where we have used the Iverson bracket, $[P]$ which returns the value $1$ if the logical proposition $P$ is satisfied, and returns $0$ otherwise, and the function $c_k(n)$ is defined as
	\[
	c_k(n)=\sum_{j=1}^{\lfloor n/k\rfloor}jQ\left(\frac{n-kj}{2}\right),
	\]
	and $Q(m)$ denotes the number of partitions of $m$ into distinct parts. Here $Q(x)=0$ if $x\notin \mathbb{N}$.
\end{theorem}
The following are two easy corollaries of the above theorem.
\begin{corollary}
For any positive integers $k$, $\ell$ and $n$, we have
\[
    	(-1)^{\ell-1}\left(\sum_{j=0}^{2\ell-1} (-1)^{\frac{j(j+1)}{2}} b_k(n-j(j+1)/2)-\frac{1+(-1)^{[n \equiv 0 \pmod k]+1}}{2}\cdot c_k(n) \right)\geq 0.
\]
\end{corollary}
\begin{corollary}
For positive integers $n$ and $k$, we have
\[
\sum_{j=0}^{\infty} (-1)^{\frac{j(j+1)}{2}} b_k(n-j(j+1)/2)=\frac{1+(-1)^{[n \equiv 0 \pmod k]+1}}{2}\cdot c_k(n).
\]
\end{corollary}
From the above theorem and corollaries, the following results of Merca and Yee \cite{MercaYee} follow as corollaries. Here we have corrected the exponent of the $-1$ inside the summation in the left hand side, which is $\dfrac{j(j+1)}{2}$, but was mentioned as $j$ by Merca and Yee \cite{MercaYee}.
\begin{corollary}[Theorem 1.7, \cite{MercaYee}]
	For any positive integer $\ell$ and $n$, we have
	\[
		  	(-1)^{\ell-1}\left(\sum_{j=0}^{2\ell-1} (-1)^{\frac{j(j+1)}{2}} b(n-j(j+1)/2)-\frac{1+(-1)^n}{2}\cdot c\left(\frac{n}{2}\right) \right)=\sum_{j=1}^{\lfloor n/2\rfloor}c(j)MP_\ell(n-2j),  
	\]
	where $c(n)$ is the number of subsets of $\{1, 2, \ldots, n\}$ which contains a number that is greater than the sum of the other numbers in the subset.
\end{corollary}
\begin{proof}
We notice that 
\[
c_2(2n)=\sum_{j=1}^njQ(n-j)=\sum_{m=0}^{n-1}(n-m)Q(m),
\]which was shown to be equal to $c(n)$ in the proof of Theorem 4.1 in Merca and Yee's \cite{MercaYee} work. So, we have $c_2(n)=c\left(\frac{n}{2}\right)$. Putting $k=2$ in Theorem \ref{Thm:gen1.7} we get the result. 
\end{proof}

\begin{corollary}[Corollary 4.2, \cite{MercaYee}]
Let $\ell$ and $n$ be positive integers, then we have
	\[
	(-1)^{\ell-1}\left(\sum_{j=0}^{2\ell-1} (-1)^{\frac{j(j+1)}{2}} b(n-j(j+1)/2)-\frac{1+(-1)^n}{2}\cdot c\left(\frac{n}{2}\right) \right)\geq 0.
	\]
\end{corollary}
\begin{corollary}[Corollary 4.3, \cite{MercaYee}]
Let $n$ be a positive integer, then we have
	\[
\sum_{j=0}^{\infty} (-1)^{\frac{j(j+1)}{2}} b(n-j(j+1)/2)=\frac{1+(-1)^n}{2}\cdot c\left(\frac{n}{2} \right).
\]
\end{corollary}

\section{Analytical Proofs of our Main Results}\label{sec:proof1}

In this section, we prove all the theorems stated in the previous section, using analytical methods. Our proofs follow closely the techniques used by Merca and Yee \cite{MercaYee}.
\subsection{Proof of Theorems \ref{ThmGF} and \ref{ThmComb}}
We start with the generating function for partitions where the power of $z$ keeps track of parts with multiplicity $\geq k$,
\begin{multline*}
       \prod_{j=1}^\infty (1+q^j+q^{2j}+\cdots +q^{(k-1)j}+z^j(q^{kj}+q^{(k+1)j}+\cdots))\\ =\prod_{j=1}^\infty \left(\frac{q^{kj}-1}{q^j-1}+z^j\frac{q^{kj}}{1-q^j} \right)
    =\frac{1}{\qq}\prod_{j=1}^\infty (1+(z^j-1)q^{kj}). 
\end{multline*}
Now, taking the derivative w.r.t. $z$ and setting $z\rightarrow 1$ we get,
\[
\sum_{n=1}^\infty b_k(n)q^n=\frac{1}{\qq}\sum_{j=1}^\infty jq^{kj}=\frac{1}{\qq}\cdot \frac{q^k}{(1-q^k)^2}.
\]

In a similar way, we have
\begin{multline*}
        \sum_{n=1}^\infty a_{k,p}(n)q^n\\
    =\frac{\partial}{\partial z} (1-q^p+z^pq^p) \prod_{j=1}^\infty ((1+q^j+q^{2j}+\cdots) -(q^{kj+p}+q^{(k+1)j+p}+\cdots) \\+z^{kj+p}(q^{kj+p}+q^{(k+1)j+p}+\cdots))\mid_{z=1}.
\end{multline*}
We subtract $(q^{kj+p}+q^{(k+1)j+p}+\cdots)$ from the first term, because we count the parts of the type $kj+p$ where $j\geq1$ in the third term, and we multiply by $(1-q^p+z^pq^p)$ because of the parts of the form $kj+p$ where $j=0$. Therefore,
   \begin{align*}
   	\sum_{n=1}^\infty a_{k,p}(n)q^n
    &=~\frac{\partial}{\partial z} (1-q^p+z^pq^p) \prod_{j=1}^\infty \left( \frac{1-q^{kj+p}}{1-q^j}+z^{kj+p}\frac{q^{kj+p}}{1-q^j} \right)\bigg|_{z=1}\\
    &=~\frac{1}{\qq}\frac{\partial}{\partial z} \prod_{j=0}^\infty \left( 1+(z^{kj+p}-1)q^{kj+p}\right)\bigg|_{z=1}\\
   	&=~\frac{1}{\qq}\sum_{j=0}^\infty (kj+p)q^{kj+p}\\
    &=~\frac{1}{\qq} \bigg(kq^p\frac{q^k}{(1-q^k)^2}+pq^p\frac{1}{1-q^k}\bigg) \\
    &=~\frac{1}{\qq}\cdot \frac{pq^p+(k-p)q^{p+k}}{(1-q^k)^2}.
\end{align*}

The case $p=0$ in the above will give us the generating function for $a_k(n)$.

Theorem \ref{ThmComb} immediately follows from Theorem \ref{ThmGF}; we just compare coefficients.

\subsection{Proof of Theorem \ref{Thm:trunc}}
The generating function for $M_\ell(n)$ was found by Andrews and Merca \cite{AndrewsMerca1}, when they studied a truncated version of Euler's pentagonal number theorem
\begin{equation}\label{eq-1}
    \frac{(-1)^{\ell-1}}{\qq}\sum_{n=-(\ell-1)}^\ell (-1)^n q^{n(3n-1)/2}=(-1)^{\ell-1}+\sum_{n=\ell}^\infty\frac{q^{\binom{\ell}{2}+(\ell+1)n}}{(q;q)_n}\gauss{n-1}{\ell-1},
\end{equation}
where $\ell \geq 1$, $(a;q)_n=\dfrac{(a;q)_\infty}{(aq^n;q)_\infty}$ and the Gausssian binomial $\gauss{n}{\ell}$ equals $\dfrac{(q;q)_n}{(q;q)_\ell(q;q)_{n-\ell}}$ whenever $0\leq \ell \leq n$ and is $0$ otherwise. The sum on the right hand side of equation \eqref{eq-1} is the generating function of $M_\ell(n)$, that is
\begin{equation}\label{eq-2}
    \sum_{n=0}^\infty M_\ell(n)q^n=\sum_{n=\ell}^\infty\frac{q^{\binom{\ell}{2}+(\ell+1)n}}{(q;q)_n}\gauss{n-1}{\ell-1}.
\end{equation}

We now multiply both sides of equation \eqref{eq-1} by
\[
\sum_{n=0}^\infty nq^{kn}=\frac{q^k}{(1-q^k)^2},
\]
which gives us (after using equation \eqref{eq-2}),
\begin{multline*}
      (-1)^{\ell-1}\left(\left(\sum_{n=1}^\infty b_k(n)q^n\right)\left(\sum_{n=-(\ell-1)}^\ell (-1)^n q^{n(3n-1)/2}\right)-\sum_{n=0}^\infty nq^{kn}\right)\\=\left(\sum_{n=0}^\infty nq^{kn}\right)\left(\sum_{n=0}^\infty M_\ell(n)q^n\right).
\end{multline*}
Using the Cauchy product of two power series, Theorem \ref{Thm:trunc} follows from the above.

\subsection{Proof of Theorem \ref{Thm:gen1.7}} The generating function for $MP_\ell(n)$ was found by Andrews and Merca \cite{AndrewsMerca2} when they considered a truncated theta identify of Gauss,
\begin{equation}\label{eq:3.1}
        \frac{(-q;q^2)_\infty}{(q^2;q^2)_\infty}\sum_{j=0}^{2\ell-1}(-q)^{j(j+1)/2}
    =1+(-1)^{\ell-1}\frac{(-q;q^2)_\ell}{(q^2;q^2)_{\ell-1}}\sum_{j=0}^\infty \frac{q^{\ell(2\ell+2j+1)}(-q^{2\ell+2j+3};q^2)_\infty}{(q^{2\ell+2j+2};q^2)_\infty}.
\end{equation}
The sum on the right hand side of equation \eqref{eq:3.1} is the generating function of $MP_\ell(n)$.

We now multiply both sides of equation \eqref{eq:3.1} by $\dfrac{q^k}{(1-q^k)^2}\cdot (-q^2;q^2)_\infty$ and deduce the following identity
\begin{multline}\label{eq:pj}
(-1)^{\ell-1}\left(\left(\sum_{n=0}^\infty b_k(n)q^n\right)\left(\sum_{n=0}^{2\ell-1} (-q)^{n(n+1)/2}\right)- \dfrac{q^k}{(1-q^k)^2}\cdot (-q^2;q^2)_\infty \right)\\=\left(\dfrac{q^k}{(1-q^k)^2}\cdot (-q^2;q^2)_\infty\right)\left(\sum_{n=0}^\infty MP_\ell(n)q^n\right).
\end{multline}
We know that the generating function of the number of partitions into distinct parts is \[ \sum_{n=0}^\infty Q(n)q^n=(-q;q)_\infty.\] Using this, we have
\begin{align*}
\dfrac{q^k}{(1-q^k)^2}\cdot (-q^2;q^2)_\infty &=~\sum_{n=0}^{\infty}nq^{kn}\cdot \sum_{m=0}^{\infty}Q(m)q^{2m}\\
&=~\sum_{n=0}^{\infty}\sum_{j=1}^{\lfloor n/k\rfloor} j Q\left(\frac{n-kj}{2}\right)q^n\\
&=~\sum_{n=0}^\infty c_k(n)q^n.
\end{align*}

Putting this in equation \eqref{eq:pj} we get,
\begin{multline*}
(-1)^{\ell-1}\left(\left(\sum_{n=0}^\infty b_k(n)q^n\right)\left(\sum_{n=0}^{2\ell-1} (-q)^{n(n+1)/2}\right)-\sum_{n=0}^\infty c_k(n)q^n \right)\\=\left(\sum_{n=0}^\infty c_k(n)q^n\right)\left(\sum_{n=0}^\infty MP_\ell(n)q^n\right).
\end{multline*}
Using the Cauchy product of two power series, Theorem \ref{Thm:gen1.7} follows from the above.

\section{Combinatorial Proofs of some of our Results}\label{sec:comb}
In this section we give combinatorial proofs of all but one (Theorem \ref{Thm:gen1.7}) of our results. The approach again closely follows that of Merca and Yee \cite{MercaYee}.
\subsection{Proof of Theorem \ref{ThmComb}} For part (1), we note that for a partition $\la \vin n$, if $ka$ is a part of $\la$ then we split this part into $k$ $a$'s while keeping the remaining parts of $\la$ unchanged. Let us call the new partition $\mu$, then clearly the part $a$ has multiplicity at least $k$ in $\mu$, so we get
\begin{align*}
    a_k(n)&=~\sum_{\la \vin n}\text{different parts divisble by}~k\\
    &=~ k\sum_{\mu\vin n}\text{differents parts with multiplicity}\geq k     =kb_k(n).
\end{align*}

For part (2), let $ka+p$ be a part of $\la \vin n$ which is counted in $a_{k,p}(n)$. We now split $ka$ into $k$ $a$'s while keeping the remaining parts unchanged to get a new partition $\mu \vin n-p$. Again, we split $(ka+p)+(k-p)$ into $k$ $(a+1)$'s while keeping the remaining parts unchanged to get a new partition $\nu \vin n+k-p$. We have
\begin{align*}
      a_{k,p}(n)=&~\sum_{\la \vin n}\text{different parts}\equiv p \pmod k\\
    =&~(k-p)\sum_{\mu \vin n-p}\text{different parts with multiplicity}\geq k \\&+p\sum_{\nu \vin n+k-p}\text{different parts with multiplicity}\geq k\\
    =&~(k-p)b_k(n-p)+pb_k(n+k-p).
\end{align*}

\subsection{Proof of Theorem \ref{ThmGF}}
We prove the generating function for $a_k(n)$ here; the other two generating functions can be proved combinatorially by combining the previous subsection with this proof. In fact, our proof is the same when $2$ is replaced by $k$ in the proof of Corollary \ref{CorM} given by Merca and Yee \cite{MercaYee}, so for the sake of brevity we just outline the steps.

We work with two sets of overpartitions, let $\bar P_k(n)$ be the set of overpartitions of $n$ where exactly one part divisible by $k$ is overlined, and let $\bar A_k(n)$ be the set of colored overpartitions of $n$ where exactly one part divisible by $k$ is overlined and at most one other part divisible by $k$ is colored with blue color. For instance, we have
\[
\bar P_3(6)=\{\bar 6, \bar 3+3, \bar 3+2+1, \bar 3+1+1+1\},
\]
and
\[
\bar A_3(6)=\{\bar 6, \bar 3+3, \bar 3+\textcolor{blue}{3}, \bar 3+2+1, \bar 3+1+1+1\}.
\]
Clearly, $\bar P_k(n)$ is a subset of $\bar A_k(n)$, and we have
\begin{equation}\label{eq:p1}
    a_k(n)=\sum_{\la \in \bar P_k(n)}\text{the overlined part of}~\la .
\end{equation}

Also note that for each partition in $\bar A_k(n)$ we can decompose it into a tuple $(\la,\mu, \nu)$ where $\la$ is the overline part, $\mu$ is the colored part and $\nu$ are the non-colored parts. This gives us
\begin{equation}\label{eq:p2}
    \sum_{n\geq 0}\bar A_k(n)q^n=\frac{q^k}{(1-q^k)}\cdot \frac{1}{(1-q^k)}\cdot \frac{1}{\qq}.
\end{equation}

We now set up the following surjection from $\bar A_k(n)$ to $\bar P_k(n)$: if there is a colored part, we merge it with the overlined part to get a resulting overlined part. The new partition is clearly in $\bar P_k(n)$, and if there are no colored parts then we keep the partition unchanged. Now, for an overlined part $\overline{ka}$ of $\mu\in \bar P_k(n)$, there are $a$ ways to merge an overlined part with a colored part to get $ka$, so we have
\begin{equation}\label{eq:p3}
    \sum_{\mu \in \bar P_k(n)}\text{the overline part of}~\mu=\sum_{\nu \in \bar A_k(n)}k.
\end{equation}

From equations \eqref{eq:p1}, \eqref{eq:p2} and \eqref{eq:p3} we get
\[
\sum_{n\geq 0}a_k(n)q^n=\frac{1}{\qq}\cdot \frac{kq^k}{(1-q^k)^2}.
\]

\subsection{Proof of Theorem \ref{Thm:trunc}}
Again, our proof is similar to the proof of Corollary \ref{cor-m-t}, given by Merca and Yee \cite{MercaYee}, so we mention the main steps without going into too much details. Theorem \ref{Thm:trunc} is equivalent to the following
\begin{multline*}
    (-1)^{\ell-1}\left(\sum_{j=-(\ell-1)}^\ell (-1)^j \left(\sum_{\la \in \bar P_k(n-j(3j-1)/2}\text{overlined part of }\la \right)\right.\\-\left.\frac{1+(-1)^{[n \equiv 0 \pmod k]+1}}{2}\cdot n \right)=\sum_{j=1}^{\lfloor n/k\rfloor}kjM_\ell(n-kj),
\end{multline*}
where we have used Theorem \ref{ThmComb} and equation \eqref{eq:p1}.

We note that
\[
\sum_{\la \in \bar P_k(n)}\text{overlined part of }\la=\sum_{m=1}^{\lfloor n/k\rfloor}km\sum_{\mu \vin (n-km)}1=\sum_{m=1}^{\lfloor n/k\rfloor}km \cdot p(n-km).
\]
The above equation is true since any partition $\la \in \bar P_k(n)$ can be made into a pair of partitions $(\nu, \mu)$ where $\nu$ is the overlined part and $\mu$ is then an ordinary partition.

So, we get
\begin{multline*}
    (-1)^{\ell-1}\sum_{j=-(\ell-1)}^\ell (-1)^j \left(\sum_{\la \in \bar P_k(n-j(3j-1)/2}\text{overlined part of }\la \right)\\
    =(-1)^{\ell-1}\sum_{j=-(\ell-1)}^\ell (-1)^j \sum_{m=1}^{\lfloor (n-j(3j-1)/2)/2\rfloor}km \cdot p(n-j(3j-1)/2-km).
\end{multline*}
The above is equivalent to
\begin{equation}\label{pf-t-1}
    \sum_{m=1}^{\lfloor n/k\rfloor}km\left((-1)^{\ell-1}\sum_{j=-(\ell-1)}^\ell (-1)^j p(n-km-j(3j-1)/2) \right),
\end{equation}
where we rearrange the summation and take $p(n)=0$ if $n<0$.

Merca and Yee \cite{MercaYee} have given a combinatorial proof of the truncated pentagonal number theorem, which is equivalent to the following identity
\begin{equation}\label{pf-t-2}
    (-1)^{\ell-1}\sum_{j=0}^{\ell -1})(-1)^j(p(n-j(3j+1)/2)-p(n-(j+1)(3j+2)/2))=M_\ell(n).
\end{equation}
Using equation \eqref{pf-t-2} in \eqref{pf-t-1}, we get that \eqref{pf-t-1} is equal to
\[
\sum_{m=1}^{\lfloor n/k\rfloor}km\cdot M_\ell(n-km)+(-1)^{\ell-1}\cdot \frac{1+(-1)^{[n \equiv 0 \pmod k]+1}}{2}\cdot n,
\]
which proves the result, since equation \eqref{pf-t-2} already has a combinatorial proof. We need the term $(-1)^{\ell-1}n$ when $n \equiv 0 \pmod k$ because, if $n=kr$ for some $r$ and $m=n/k$, then without this term we get
\[
(-1)^{\ell-1}np(0)=nM_k(0)\Rightarrow n=0.
\]

\section{Concluding Remarks}\label{sec:rem}

\begin{enumerate}
    \item A combinatorial proof of Theorem \ref{Thm:gen1.7} is left as an open problem. Any combinatorial proof of Theorem \ref{Thm:gen1.7} would hinge on a combinatorial interpretation of $c_k(n)$, like we have for $c_2(2n)$. So, a first step towards a combinatorial proof would be such an interpretation of $c_k(n)$.
    \item Identities of the type in Theorems \ref{Thm:trunc} and \ref{Thm:gen1.7} are also known for some other partition statistics, for instance one can see some recent work of Merca \cite{Merca}. It would be interesting to see if one can relate such partition statistics with the ones introduced in this paper.
\end{enumerate}

\section*{Acknowledgements}

The second author is partially supported by the Leverhulme Trust Research Project Grant RPG-2019-083. The authors thank Dr. Nilufar Mana Begum for bringing the paper of Merca and Yee \cite{MercaYee} to their notice, and Prof. Nayandeep Deka Baruah for encouragement. The authors also thank the editor and the anonymous referee for helpful comments.

\end{document}